\documentclass[a4paper,12pt]{article}
\usepackage{amsmath}
\usepackage{amssymb}
\usepackage{listings}

\usepackage{amsthm}
\theoremstyle{plain}
\newtheorem{theorem}{Theorem}
\newtheorem{prop}{Proposition}
\newtheorem{lemm}{Lemma}
\newtheorem{corr}{Corollary}

\theoremstyle{definition}

\def\dom{\mathrm{dom}\,}

\def\pa{\partial}
\def\ol{\overline}
\def\be{\begin{equation}}
\def\ee{\end{equation}}
\newcommand{\Ri}{\mathbb{R}\cup\{+\infty\}}
\newcommand{\R}{\mathbb{R}}
\newcommand{\N}{\mathbb{N}}
\def\eps{\varepsilon}
\def\la{\lambda}

\DeclareMathOperator*{\argmin}{-argmin}
\DeclareMathOperator*{\Crit}{-Crit}
\DeclareMathOperator*{\crit}{-crit}

\title{Slopes and Moreau-Rockafellar Theorem}
\author{M. Ivanov\thanks{Radiant Life Technologies Ltd., Nicosia, Cyprus, e-mail:milen@radiant-life-technologies.com}, N. Zlateva\thanks{{Faculty of Mathematics and Informatics, Sofia University,   5, James Bourchier Blvd, 1164 Sofia, Bulgaria, e-mail:zlateva@fmi.uni-sofia.bg}}}
\date{Dedicated to R. T. Rockafellar}

\begin{document}
\maketitle
\begin{abstract}
Properties of local and global slope of a function and its approximate critical points sets are studied in relation to determination of the function.
\\[0.2cm]
\end{abstract}

 \section*{Introduction}\label{sec:intro}
 The slopes have been studied in relation to gradient flows, e.g. \cite{AGS,GMT}, but they are well defined where the differential inclusions are problematic,
 as for example in general metric spaces lacking any linear-like structure. The evolution of the concept is excellently explained in the monograph \cite{I},
 see also \cite{DIL}. For a state-of-the-art account see \cite{Aris-new}.
 Many striking and surprising results have been obtained to show that in many cases slope alone determines the function, see e.g. \cite{PASV,TZ,Aris-new}. One drawback of this very active and recent development is the use of transfinite induction. Here we remove any such reference and in the process develop the topic a bit further.

 Why it is important to avoid the use of transfinite induction we explain briefly in the Appendix.

 In this article we consider results like: let $f$ and $g$ be continuous functions on a complete metric space and let the local slope of $f$ be everywhere bigger than that of $g$.
 The case when $f$ and $g$ are differentiable on a Banach space with $\|f'(x)\| > \|g'(x)\|$ whenever $f'(x)\neq0$, is interesting enough.
 Then under certain conditions the infimum of $f-g$ is attained on the set of approximately critical points of $f$.

 Our approach, based on an idea from \cite{Aris}, is to consider for an arbitrary fixed $x_0$ the minimisation problem
 $$
    \begin{cases}
        f(x) \to \min\\
        f(x)-g(x) \le f(x_0) - g(x_0),
    \end{cases}
 $$
 and observe that -- because of the slopes condition -- any (approximately) critical point cannot be on the boundary of the feasible set, so it is actually a critical point of $f$.

 Even more can be said in the case of global slopes, where we can estimate the infimum of $f-g$ through the infima of $f$ and $g$. This fully primal result almost immediately yields the celebrated Moreau-Rockafellar Theorem \cite{M,R}.

 The first version of this work has been completed before the appearance of the recent preprint \cite{Aris-new}.
 There a new general result is obtained via transfinite induction. Here we cover a variant of this result, see Theorem~\ref{thm:cont-case-new}.

 The paper is organised as follows. In Section~\ref{sec:def} we give the necessary definitions.  Section~\ref{sec:prop} is devoted to the properties of the slopes.
 Section~\ref{sec:res} contains the main results.  In Section~\ref{sec:mrt} we give a proof to Moreau-Rockafellar Theorem using slopes.

\section{Definitions}\label{sec:def}
Here we define some concepts and notations we will be using throughout the paper.

We consider functions defined on a  metric space $(M,\rho)$ with values in $\mathbb R\cup\{+\infty\}$.
Although it is somewhat standard in such setting to adopt the convention $\infty-\infty=\infty$, we find this approach prone to errors, thus we avoid it. So, "$\infty-\infty$" is undefined.
Therefore, if nothing more is specified, formula like for example $\{x:\ f(x)-g(x) \le \lambda\}$ means those $x$ for which at least one of $f(x)$ and $g(x)$ is finite, and
$f(x)-g(x) \le \lambda$. This said, we usually take care to write more explicitly, so for $\lambda\in\mathbb{R}$ we would rather write the equivalent formulae
$\{x\in\dom f:\ f(x)-g(x) \le \lambda\}$ where, as usual, for $f:M\to\Ri$, the domain of $f$ is
$$
    \dom f := \{x:\ f(x) < \infty\}.
$$
We denote the sub-level set of $f$ below the level $\lambda\in\mathbb{R}$ by
$$
    L_\lambda(f) := \{x:\ f(x)\le\lambda\}.
$$
Following the convention above,
$$
    L_\lambda(f-g) := \{x\in\dom f:\ f(x)-g(x)\le\lambda\}.
$$

For $f$ which is proper, i.e. such that $\dom f\neq\varnothing$, and bounded below; and $\eps \ge 0$, we denote
$$
    \varepsilon\argmin f := \{x: f(x) \le \inf f + \varepsilon\},
$$
hence, $\varepsilon\argmin f = L_{(\inf f+\varepsilon)} (f)$.

If $x\in\dom f$ then the \emph{local slope} of $f$ at $x$ is defined as
$$
    |\nabla f|(x) := \limsup_{y\to x\atop y\neq x}\frac{[f(x)-f(y)]^+}{\rho(x,y)}.
$$
Here $[t]^+ := \max\{0,t\} = (t+|t|)/2$, so $|\nabla f|(x) = 0$ if $x$ is a local minimum to $f$, and if not
$$
    |\nabla f|(x) = \limsup_{y\to x\atop y\neq x}\frac{f(x)-f(y)}{\rho(x,y)}.
$$
In the latter statement we explicit that $|\nabla f|(x) = 0$ at the isolated points of $\dom f$.
This notion was introduced by De Giorgi, Marino and Tosques~\cite{GMT} and used thereafter for the study of concepts as descent curves (e.g. \cite{DIL}),    error bounds (e.g. \cite{AC}),   subdifferential calculus and metric regularity (e.g. the monograph~\cite{I} and  references therein).

We may consider $|\nabla f|$ as a function from $\dom f$ to  $  [0,\infty]$. Moreover,
\begin{equation}
    \label{eq:mult-slope}
    |\nabla (rf)|(x) = r |\nabla f|(x),\quad\forall r\ge 0,\ \forall x\in\dom f.
\end{equation}
Because $[a+b]^+ \le [a]^+ + [b]^+$ we have
\begin{equation}
    \label{eq:triang-ineq}
    |\nabla (f+g)| \le |\nabla f| + |\nabla g|,
\end{equation}
which  for $g:M\to\mathbb{R}$ implies
\begin{equation}
    \label{eq:triang-ineq-prim}
    |\nabla (f-g)| \ge |\nabla f| - |\nabla g|,\quad\forall x\in\dom|\nabla g|.
\end{equation}
Let us denote
$$
    \varepsilon\crit f := \{x:\ |\nabla f|(x) \le \varepsilon\}.
$$
In other words, $\varepsilon\crit f = L_\varepsilon(|\nabla f|)$. If the space $M$ is complete and the function $f$ is lower semicontinuous and bounded below, then Ekeland Variational Principle gives that $\inf |\nabla f| = 0$, thus $\varepsilon\crit f = \varepsilon\argmin |\nabla f|$ (see Proposition~\ref{pro:EVP} below).

The \emph{global slope} of $f$ at $x\in\dom f$ is defined as
$$
    |\widetilde \nabla f|(x) := \sup_{y\neq x}\frac{[f(x)-f(y)]^+}{\rho(x,y)},
$$
and we can consider $|\widetilde \nabla f|$ as a function from $\dom f$ to  $  [0,\infty]$.  This notion is introduced by Thibault and Zagrodny in~\cite{TZ}.

Obviously, $|\widetilde \nabla f| \ge |\nabla f|$ on $\dom f$. It is also clear that the subadditivity \eqref{eq:mult-slope}, \eqref{eq:triang-ineq} and \eqref{eq:triang-ineq-prim},
 holds for the global slope $|\widetilde\nabla\cdot|$ as well.

By analogy we define
$$
    \varepsilon\Crit f := L_\varepsilon(|\widetilde \nabla f|).
$$
It is immediate  that
$$
    \varepsilon\Crit f = \{x\in\dom f:\ f(y)\ge f(x) - \varepsilon\rho (y,x),\ \forall y\}.
$$

\section{Properties of the slopes}\label{sec:prop}
In this section we consider some properties of the local and the global slope of a given function.
\begin{lemm}
    \label{lem:log-ro}
    Let $(M,\rho)$ be a metric space and let $a\in M$ be fixed. Set
    $$
        \varphi(x) := -\log\rho(x,a).
    $$
    Then
    \begin{equation}
        \label{eq:log-ro}
        |\widetilde\nabla\varphi|(x) \le \frac{1}{\rho(x,a)},\quad\forall x\neq a.
    \end{equation}
\end{lemm}
\begin{proof}
    Fix $x\neq a$ and let $y\neq a$ be such that $\varphi(y)<\varphi(x)$, that is, $\rho(y,a)>\rho(x,a)$. Using the inequality
    $$
        \log(t+s) \le \log t + \frac{s}{t},\quad \forall t,s > 0,
    $$
    due to the concavity of the logarithm, we get
    $$
        \log\rho(y,a) \le \log(\rho(x,a)+\rho(x,y))
        \le \log \rho(x,a) + \rho(x,y)/\rho(x,a).
    $$
    Therefore, $\varphi(y) \ge \varphi(x) - \rho(x,y)/\rho(x,a)$, that is,
    $$
        \frac{\varphi(x)-\varphi(y)}{\rho(x,y)} \le \frac{1}{\rho(x,a)}.
    $$
    As $y$ with $\varphi(y) < \varphi(x)$ was arbitrary, we get \eqref{eq:log-ro}.
\end{proof}

\begin{lemm}
    \label{lem:min}
    Let $(M,\rho)$ be a metric space. Let $g:M\to \mathbb{R}\cup\{+\infty\}$ be a proper function. Let $\lambda\in \mathbb{R}$. Consider the function
    $$
        g_1(x) := \min\{g(x),\lambda\},\quad\forall x\in M.
    $$
    Then
    $$
    |\nabla g_1|(x) \le |\nabla g| (x)\text{ and }|\widetilde\nabla g_1|(x) \le |\widetilde\nabla g| (x),\quad\forall x\in \dom g.
    $$
\end{lemm}
\begin{proof}
    Fix $x\in \dom g$.

    It is enough to show that
    \begin{equation}
        \label{eq:g-1+ineq}
        [g_1(x)-g_1(y)]^+ \le [g(x)-g(y)]^+, \quad\forall y\in M,
    \end{equation}
    because then we can just take $\limsup$ or $\sup$.

    \textsc{Case 1:}  $g(x)\le \lambda$.

    Then $g_1(x) = g(x)$ and so $[g_1(x)-g_1(y)]^+=[g(x)-g_1(y)]^+$ for any $y\in M$.

    If $y$ is such that $g_1(y) \ge g(x)$ then also $g(y) \ge g(x)$ and  $[g(x)-g_1(y)]^+=[g(x)-g(y)]^+=0$.

    If, on the other hand, $y$ is such that $g_1(y) < g(x)$ then, of course, $g_1(y) = g(y)$. So,
    $[g(x)-g_1(y)]^+ = [g(x)-g(y)]^+$.

    We see that in this case \eqref{eq:g-1+ineq} holds even with equality.

    \textsc{Case 2:} $g(x) > \lambda$.

    Then $g_1(x) = \lambda$ and so $[g_1(x)-g_1(y)]^+=[\lambda-g_1(y)]^+$.

    If $y$ is such that $g(y)\ge \lambda$ then by definition $g_1(y) = \lambda$ and $[g_1(x)-g_1(y)]^+=0\le[g(x)-g(y)]^+$.

    If $y$ is such that $g(y) < \lambda$ then $g_1(y) = g(y)$ and $[g_1(x)-g_1(y)]^+=[\lambda-g(y)]^+$. Since  $t\to[t]^+$ is a non-decreasing function,  $[\lambda-g(y)]^+\le [g(x)-g(y)]^+$.

    So, \eqref{eq:g-1+ineq} holds.
\end{proof}

 Note that $\varepsilon\Crit f$ is the set of those points where  the Hausdorff regularisation of $f$ coincides with $f$:
$$
    f(x) = f_{1/\varepsilon} (x) := \inf \{f(y) +\varepsilon \rho(y,x)\}.
$$
As the Hausdorff regularisation is Lipschitz, clearly a lower semicontinuous  $f$ is Lipschitz on the closed set $\varepsilon\Crit f$, but this can also be  proved directly:

\begin{lemm}\label{lem:crit}
    Let $f:M\to\mathbb{R}\cup\{+\infty\}$ be a proper and lower semicontinuous function on a metric space $(M,\rho)$. Then $\varepsilon\Crit f$ is a closed set on which $f$ is $\varepsilon$-Lipschitz.
\end{lemm}

\begin{proof}
Fix  $\varepsilon > 0 $ and a convergent sequence $\{x_n\}_{n=1}^\infty\subset\varepsilon\Crit f$. Let
$$
    \lim_{n\to\infty} x_n = \bar x.
$$
By definition
$$
    f(\bar x) \ge f(x_n) - \varepsilon\rho(\bar x,x_n),
$$
which, after taking a limit, gives
$$
    f(\bar x) \ge \limsup_{n\to\infty} f(x_n).
$$
But $f$ is given to be lower semicontinuous at $\bar x$, so
$$
    \lim_{n\to\infty} f(x_n) = f(\bar x).
$$
Let $y\in\dom f$ be arbitrary. By definition,
$$
    f(y) \ge f(x_n) -\varepsilon\rho(y,x_n),
$$
so
$$
    f(y) \ge f(\bar x) -\varepsilon\rho(y,\bar x).
$$
In particular $\bar x\in\dom f$ and, therefore, $\bar x\in \varepsilon\Crit f$.

We have proved that $\varepsilon\Crit f$ is closed.

Let $x,y\in \varepsilon\Crit f$. By definition, $f(y) - f(x) \ge -\varepsilon\rho(y,x)$ and $f(x) - f(y) \ge -\varepsilon\rho(x,y)$, so $|f(y) - f(x)| \le \varepsilon\rho(y,x)$.
\end{proof}

\begin{lemm}\label{lem:invar}
    Let $f,g:M\to\mathbb{R}\cup\{+\infty\}$ and let for some $\lambda \in \R$,
    \begin{equation}
        \label{eq:L-lam}
          L_\lambda(f-g)\cap \dom f \cap\dom  g\neq \varnothing.
    \end{equation}
Let $M_1 := L_\lambda(f-g)$, and
$$
        f_1 := f\restriction_{M_1}.
    $$
    If $x\in M_1$ and
    \begin{equation}
        \label{eq:main-ineq}
        |\nabla  f|(x) > |\nabla g|(x),
    \end{equation}
then
    $$
        |\nabla f_1|(x) = |\nabla f|(x).
    $$
\end{lemm}
\begin{proof}
Take $x\in M_1$ satisfying  \eqref{eq:main-ineq}.   Since $|\nabla f|(x) > 0$, by definition there is a sequence $y_n\to x$ such that
    \begin{equation}
        \label{eq:seq*}
        \lim_{n\to\infty} \frac{f(x)-f(y_n)}{\rho(x,y_n)} = |\nabla f|(x).
    \end{equation}
    Let $r\in\mathbb{R}$ be such that
    $$
        |\nabla f|(x) > r > |\nabla g|(x).
    $$
    By definition,
    $$
        \frac{g(x)-g(y_n)}{\rho(x,y_n)} < r
    $$
    for all $n$ large enough. By \eqref{eq:seq*}
    $$
        \frac{f(x)-f(y_n)}{\rho(x,y_n)} > r
    $$
    for all $n$ large enough. So,
    $$
        \frac{g(x)-g(y_n)}{\rho(x,y_n)} < r < \frac{f(x)-f(y_n)}{\rho(x,y_n)},\quad\forall n>K.
    $$
    This implies
    $$
        f(y_n) - g(y_n) < f(x) - g(x) \le \lambda ,\quad\forall n>K.
    $$
    That is, $y_n\in M_1$ for all $n>K$. Then \eqref{eq:seq*} becomes
    $$
    \lim_{n\to\infty} \frac{f_1(x)-f_1(y_n)}{\rho(x,y_n)} = |\nabla f|(x),
    $$
    meaning that $|\nabla  f_1|(x) \ge |\nabla  f|(x)$, while the inverse inequality is trivial.
\end{proof}

Of course, Lemma~\ref{lem:invar} has a variant for global slopes with a shorter proof. We give the details for completeness.
\begin{lemm}
    \label{lem:invar-tilde}
    Let $f,g:M\to\mathbb{R}\cup\{+\infty\}$ and let \eqref{eq:L-lam} be satisfied for some $\lambda \in \R$.
Let $M_1 := L_\lambda(f-g)$, and let
$$
        f_1 := f\restriction_{M_1}.
    $$
    If $x\in M_1$ and
    \begin{equation}
        \label{eq:main-ineq-tilde}
        |\widetilde\nabla  f|(x) > |\widetilde\nabla g|(x),
    \end{equation}
then
    $$
        |\widetilde\nabla  f_1|(x) = |\widetilde\nabla f(x)|.
    $$
\end{lemm}
\begin{proof}
    Fix $x\in M_1$ such that \eqref{eq:main-ineq-tilde} holds. Take any $r  > 0$ such that
    $$
        |\widetilde\nabla g|(x) < r < |\widetilde\nabla f|(x).
    $$
    By definition, there is $y\in M$ such that $f(x)-f(y) > r\rho(x,y)$. On the other hand, $g(x)-g(\cdot)\le r\rho(x,\cdot)$, so, in particular, $g(x)-g(y)\le r\rho(x,y)$. Thus, $f(x)-f(y) > g(x)-g(y)$ and $f(y)-g(y) < f(x)-g(x) \le\lambda$ meaning that $y\in M_1$ and so $f_1(x)-f_1(y) > r\rho(x,y)\Rightarrow |\widetilde\nabla f_1|(x) > r$.

    Therefore, $|\widetilde\nabla f_1|(x) \ge |\widetilde\nabla f|(x)$.
\end{proof}

Let us recall the famous Ekeland Variational Principle, see e.g. \cite[p. 45]{Ph}. Let $(M,\rho)$ be a complete metric space and let $f:M\to \Ri$ be a proper lower semicontinuous and bounded below function. Let $\eps >0$ and $x_0\in \eps\argmin f$. Then for every $\la >0$ there exists $x_\la $ such that
\begin{itemize}
\item[(i)] $\la \rho(x_0,x_\la)\le f(x_0)-f(x_\la)$;
\item[(ii)] $\la \rho(x,x_\la)+f(x)>f(x_\la),\quad \forall x\neq x_\la$;
\item[(iii)]  $\rho(x_0,x_\la)\le \eps/\la$.
\end{itemize}

Here we will use a variant of the Ekeland Variational Principle, namely

\begin{prop}\label{pro:EVP} Let $(M,\rho)$ be a complete metric space.
Let $f:M\to \Ri$ be a proper lower semicontinuous and bounded below function and let   $x_0\in \dom  f$. Then for every $\la >0$ there exists $x_\la \in \la\Crit f$ such that
\[
f(x_\la)\le f(x_0)-\la \rho(x_0,x_\la).
\]
\end{prop}

\begin{proof}
If $f$ attains its minimum at $x_0$, then take $x_\la=x_0$. Otherwise, set $\eps:= f(x_0)-\inf f$ and apply the Ekeland Variational Principle at $x_0$ for $\la>0$ to get $x_\la$. Then (ii) means that $x_\la\in\la\Crit f$, while (i) is the claimed inequality.
\end{proof}
The graphical density of  $\dom |\widetilde\nabla f|$ in  $\dom f$ is pretty standard. We sketch a proof for reader's convenience.
\begin{lemm}
    \label{lem:gr-dens}
    Let $(M,\rho)$ be a complete metric space and let $f:M\to\mathbb{R}\cup\{+\infty\}$ be proper lower semicontinuous and bounded below. Then for each $x\in\dom f$ there is a sequence $\{x_n\}_{n=1}^\infty\subset \dom|\widetilde\nabla f|$ such that
    \begin{equation}
        \label{eq:gr-dens}
        (x,f(x)) := \lim_{n\to\infty}(x_n,f(x_n)).
    \end{equation}
\end{lemm}
\begin{proof}
    By subtracting $\inf f$ from $f$ we may assume that $\inf f = 0$.

    Let $\bar x\in\dom f$. Let $\lambda_n\to\infty$. By Proposition~\ref{pro:EVP} there are $x_n\in\lambda_n\Crit f\subset \dom|\widetilde\nabla f|$ such that $f(x_n) \le f(\bar x) - \lambda_n\rho(x_n,\bar x)$. Since $f(x_n) \ge 0$, we have $\rho(x_n,\bar x) \le f(\bar x) / \lambda_n \to 0$, as $n\to\infty$.

    Also, $f(x_n) \le f(\bar x)$, so $\limsup_{n\to\infty} f(x_n) \le f(\bar x)$, but $\liminf_{n\to\infty} f(x_n) \ge f(\bar x)$ by the lower semicontinuity of $\bar x$.
\end{proof}
In fact, we will be rather using the following consequence of the  graphical density.

\begin{lemm}
    \label{lem:inf-on-dom}
    Let $(M,\rho)$ be a complete metric space. Let $f,g:M\to\mathbb{R}\cup\{+\infty\}$ be lower semicontinuous and let $f$ be also proper and bounded below. Then
    \begin{equation}
        \label{eq:inf-on-dom}
        \inf_{\dom f} (f-g) = \inf _{\dom|\widetilde\nabla f|} (f-g).
    \end{equation}
\end{lemm}
\begin{proof}
    Let $r\in\mathbb{R}$ be arbitrary such that
    $$
        r > \inf_{\dom f} (f-g).
    $$
    Let $\bar x\in\dom f$ be such that $r > f(\bar x) - g(\bar x)$. By Lemma~\ref{lem:gr-dens} there is a sequence $\{x_n\}_{n=1}^\infty\subset \dom|\widetilde\nabla f|$ such that \eqref{eq:gr-dens} is fulfilled for $x=\bar x$.

    By the lower semicontinuity of $g$
    $$
        g(\bar x) \le \liminf _{n\to\infty} g(x_n) \Rightarrow -g(\bar x) \ge \limsup_{n\to\infty} (-g(x_n)).
    $$
    Therefore, $f(\bar x) - g(\bar x) \ge \limsup _{n\to\infty} (f(x_n) - g(x_n))$, that is, $r > f(x_n) - g(x_n)$ for all large enough $n$.
\end{proof}
\begin{lemm}
    \label{lem:invar-lsc}
    Let $(M,\rho)$ be a complete metric space.
    Let $f,g:M\to\mathbb{R}\cup\{+\infty\}$ and let for some $\lambda \in \R$ and $s>0$
    \begin{equation}
        \label{eq:L-lam-lsc}
          L_\lambda(f-g)\cap s\Crit f \cap\dom g\neq \varnothing.
    \end{equation}
Let $M_1 := L_\lambda(f-g)\cap s\Crit f$, and
$$
        f_1 := f\restriction_{M_1}.
    $$
    If $x\in M_1$ and \eqref{eq:main-ineq-tilde} holds at $x$
then
    $$
        |\widetilde\nabla  f_1|(x) = |\widetilde\nabla f|(x).
    $$
\end{lemm}
\begin{proof}
    Fix $x\in M_1$ such that \eqref{eq:main-ineq-tilde} holds. Take any $r  > 0$ such that
    $$
        |\widetilde\nabla g|(x) < r < |\widetilde\nabla f|(x).
    $$
    In particular, because $x\in s\Crit f$,
    \begin{equation}
        \label{eq:r<s}
        r < s.
    \end{equation}
    By definition, there is $y\in M$ such that $f(x)-f(y) > r\rho(x,y)$. By Proposition~\ref{pro:EVP} applied at $y$ there is $z\in r\Crit f$ such that $f(z) \le f(y) - r\rho(y,z)$. The triangle inequality gives
    \begin{equation}
        \label{eq:z-ineq}
        f(x) - f(z) > r\rho(x,z).
    \end{equation}
    Since by definition $g(x) - g(z) < r\rho(x,z)$, as in the proof of Lemma~\ref{lem:invar-tilde} we derive that $z\in L_\lambda(f-g)$.

    But also $z\in r\Crit f$ and \eqref{eq:r<s} gives that $z\in s\Crit f$ and so, \eqref{eq:z-ineq} becomes
    $$
        f_1(x) - f_1(z) > r\rho(x,z),
    $$
    meaning that $|\widetilde\nabla f_1|(x) \ge |\widetilde\nabla f|(x)$.
\end{proof}

\section{Results}\label{sec:res}
We start with a result of Daniilidis and Salas \cite{Arisalas}. Recently it has been    independently  proved  without transfinite induction in \cite{Aris-new}.

\begin{prop}
    \label{pro:compact}
    Let $(M,\rho)$ be a metric space and let $\tau$ be some topology on $M$. Assume that $f,g:M\to\mathbb{R}$ are $\tau$-continuous and such that
    \begin{equation}\label{eq:strict-compare-compact}
        |\nabla f|(x) > |\nabla g|(x),\quad\forall x\not\in 0\crit f.
    \end{equation}
    If $f$ has $\tau$-compact level sets then
    $$
        \inf(f-g) = \inf_{0\crit f} (f-g).
    $$
\end{prop}
\begin{proof}
    Let $x_0$ be arbitrary and let $\lambda := f(x_0)-g(x_0)$. Let
    $$
        M_1 := L_\lambda(f-g),\quad f_1 := f\restriction _{M_1}.
    $$
    It is clear that $M_1$ is $\tau$-closed and $f_1$ has $\tau$-compact level sets. Let
    $$
        \bar x\in\mathrm{argmin}\, f_1,
    $$
    the latter being non-empty by Weierstrass Theorem.

    If $\bar x \not\in 0\crit f$, then \eqref{eq:strict-compare-compact} and Lemma~\ref{lem:invar} give $|\nabla f_1| (\bar x) > 0$ and, therefore, $ \bar x$ cannot be a minimum of $f_1$, contradiction.

    Thus $\bar x \in 0\crit f$ and $(f-g)(\bar x) \le \lambda = (f-g)(x_0)$.
\end{proof}

We will develop the above result in cases where completeness replaces compactness. The easiest case to work with is that of Lipschitz function because then the slopes are properly defined everywhere and one does not have to meddle with domains.
\begin{prop}
    \label{pro:lips-f-g}
    Let $(M,\rho)$ be a complete metric space and let $f,g:M\to\mathbb{R}$ be Lipschitz and such that $f$ is bounded below and
    \begin{equation}
        \label{eq:non-strict-comp-M-tilde}
        |\widetilde\nabla f|(x) \ge |\widetilde\nabla g|(x),\quad\forall x\in M.
    \end{equation}
    Then
    \begin{equation}
        \label{eq:inf-f-g-eps-crit}
        \inf (f-g) = \inf_{\varepsilon\Crit f} (f-g),\quad\forall\varepsilon>0.
    \end{equation}
\end{prop}
\begin{proof}
    By adding a constant to both $f$ and $g$ we may assume that
    $$
        f \ge 0.
    $$
    Let $x_0\in M$ and $\varepsilon>0$ be arbitrary. Set
    $$
        g_1(x) := \min\{g(x_0),g(x)\}.
    $$
    Consider
    $$
        h(x) := (1+\varepsilon) f(x) - g_1(x).
    $$
    From \eqref{eq:mult-slope} and \eqref{eq:triang-ineq} it follows that $|\widetilde\nabla h|(x) \ge (1+\varepsilon)|\widetilde\nabla f|(x) - |\widetilde\nabla g_1|(x) = \varepsilon|\widetilde\nabla f|(x) +(|\widetilde\nabla f|(x) - |\widetilde\nabla g_1|(x))$. From Lemma~\ref{lem:min} it follows that $|\widetilde\nabla f|(x) - |\widetilde\nabla g_1|(x)\ge 0$ and so
    $$
        |\widetilde\nabla h|(x) \ge \varepsilon|\widetilde\nabla f|(x),\quad\forall x\in M.
    $$
    Now $h \ge - g(x_0)$, so from Proposition~\ref{pro:EVP} there is $\bar x\in \varepsilon^2\Crit h$ such that $h(x_0)\ge h(\bar x)$. So,
    $$
        \varepsilon^2\ge |\widetilde\nabla h|(\bar x) \ge \varepsilon|\widetilde\nabla f|(\bar x)\Rightarrow \bar x\in\varepsilon\Crit f,
    $$
    and
    $$
        (1+\varepsilon)f(x_0) - g(x_0) = h(x_0) \ge h(\bar x) = (1+\varepsilon)f(\bar x) - g_1(\bar x) \ge f(\bar x) - g(\bar x),
    $$
    because $f\ge 0$ and $g_1\le g$. Since $\varepsilon>0$ was arbitrary, we  obtain
    $$
        f(x_0) - g(x_0) \ge \lim_{\varepsilon\to0}\inf_{\varepsilon\Crit f} (f-g).
    $$
    Since $x_0$ was arbitrary, we get
    $$
        \inf (f - g) \ge  \lim_{\varepsilon\to0}\inf_{\varepsilon\Crit f} (f-g),
    $$
    which is clearly equivalent to \eqref{eq:inf-f-g-eps-crit}, because $\lim_{\varepsilon\to0}\inf_{\varepsilon\Crit f} (f-g)\ge \inf_{\varepsilon\Crit f} (f-g)$ for any $\varepsilon>0$ and $\inf (f - g) \le \inf_{\varepsilon\Crit f} (f-g)$, as the infimum over the whole space is less or equal to the infimum over any non-empty subset.
\end{proof}
Of course, Proposition~\ref{pro:lips-f-g} has a version -- and that with the same proof -- for the local slope $|\nabla|$ and $\varepsilon\crit$ set. The following result in the style of Thibault and Zagrodny~\cite{TZ}, however, is $|\widetilde\nabla|$-specific.

\begin{theorem}
    \label{thm:t-z}
    Let $(M,\rho)$ be a complete metric space and let $f,g:M\to\mathbb{R}\cup\{+\infty\}$ be   proper and lower semicontinuous. Let $f$ be bounded below and such that
    \begin{equation}\label{eq:non-strict-compare-tilde}
          |\widetilde\nabla g|(x) \le |\widetilde\nabla f|(x),\quad\forall x\in\dom |\widetilde\nabla f|.
    \end{equation}

    Then
    \begin{equation}
        \label{eq:inf-33}
        \inf_{\dom f} (f-g) \ge \inf f - \inf g.
    \end{equation}
\end{theorem}
Note that, since $\displaystyle\inf_{\dom f} (f-g) < \infty$, the inequality \eqref{eq:inf-33} implies  that $\inf g \neq -\infty$. Therefore, \eqref{eq:inf-33} can be rewritten as in \cite{TZ}:
\begin{equation}
    \label{eq:inf-34}
    f(x) - \inf f \ge  g(x) - \inf g,\quad\forall x\in X.
\end{equation}
Indeed, if $f(x)=\infty$ then \eqref{eq:inf-34} is trivial; otherwise it follows from \eqref{eq:inf-33}.
\begin{proof}
    By subtracting $\inf f$ from both functions we may assume that
    \begin{equation}
        \label{eq:inf-f-0}
        \inf f = 0.
    \end{equation}
    Because of \eqref{eq:inf-on-dom} we have to show only that
    $$
        \inf_{\dom |\widetilde\nabla f|} (f-g) \ge  - \inf g.
    $$
    Assume the contrary, that is, there is $x_0 \in \dom |\widetilde\nabla f|$ such that $f(x_0) - g(x_0) < -\inf g$. Note that $x_0\in\dom g$ because of \eqref{eq:non-strict-compare-tilde}. Let $c> 1$ be such that
    $$
        \lambda := cf(x_0) - g(x_0) < -\inf g.
    $$
    Let $s:=c|\widetilde\nabla f|(x_0)$. Set
    $$
        M_1 := L_\lambda (cf-g)\cap s\Crit cf, \quad f_1 := cf\restriction_{M_1},\quad g_1 := g\restriction_{M_1}.
    $$
    Obviously, $x_0\in M_1$. From Lemma~\ref{lem:crit} we know that $s \Crit  cf$ is closed and $cf$ is Lipschitz on $s \Crit  cf$.
    Also, \eqref{eq:non-strict-compare-tilde} implies that $s \Crit  cf\subset s \Crit f\subset s \Crit  g$ and thus $g$ is also Lipschitz on $s \Crit  cf$.
    So, $(cf-g)\restriction_{s\Crit cf}$ is Lipschitz and thus $L_\lambda(cf-g)\cap s\Crit cf$ is closed.
    That is, $M_1$ is closed and $f_1$ and $g_1$ are Lipschitz.

    Since $g(x) \ge cf(x) - \lambda$ for $x\in M_1$, we have $\displaystyle \inf g_1 \ge c\inf_M f - \lambda = -\lambda > \inf g$. That is, there are $a\in M\setminus M_1$ and $r>0$ such that
    $$
        g_1(x) = g(x) > g(a) + r,\quad\forall x\in M_1.
    $$
    By definition,
    $$
        |\widetilde\nabla g|(x) \ge \frac{r}{\rho(x,a)}, \quad\forall x\in M_1.
    $$
    By \eqref{eq:non-strict-comp-M-tilde} we have
    $$
        |\widetilde\nabla f|(x) \ge \frac{r}{\rho(x,a)}, \quad\forall x\in M_1.
    $$
    Now, \eqref{eq:non-strict-compare-tilde} implies $ |\widetilde\nabla (cf)| > |\widetilde\nabla g|$ on $M_1$ and Lemma~\ref{lem:invar-lsc} applied to $cf$ and $g$ gives
    $|\widetilde\nabla f_1| (x)= c |\widetilde\nabla f|(x)$, for all $x\in {M_1}$.
    Since $c>1$ and $|\widetilde\nabla f|(x) >0$, we get $|\widetilde\nabla f_1| (x)> |\widetilde\nabla f|(x)$, for all $x\in {M_1}$.
    In particular,
    \begin{equation}
        \label{eq:nabla-f1-log-big}
        |\widetilde\nabla f_1|(x) >  \frac{r}{\rho(x,a)}, \quad\forall x\in M_1.
    \end{equation}
    Let $\varphi: M_1\to \mathbb{R}$ be defined as $\varphi(x) := -r\log\rho(x,a)$. Lemma~\ref{lem:log-ro} and \eqref{eq:nabla-f1-log-big} show that
    $$
        |\widetilde\nabla f_1|(x) > |\widetilde\nabla \varphi|(x), \quad\forall x\in M_1.
    $$
    Consider on $M_1$
    $$
        h(x) := 2f(x) - \varphi (x).
    $$
    Since $h(x) \ge r\log\rho(x,a)\ge r\log\mathrm{dist}(a,M_1) > -\infty$, the function $h$ is proper lower semicontinuous and bounded below.

    Let $\varepsilon_n\searrow0$. By Proposition~\ref{pro:EVP} applied to $h$ at $x_0$, there are $x_n\in\varepsilon_n\Crit h$ such that
    $$
        h(x_n) \le h(x_0),\quad\forall n\in\mathbb{N}.
    $$
    Since $h(x_n) \ge r\log\rho(x_n,a)$ this means that the sequence $\{\rho(x_n,a)\}_{n=1}^\infty$ is bounded.

    On the other hand, however, \eqref{eq:triang-ineq-prim} and $|\widetilde\nabla f| > |\widetilde\nabla \varphi|$ show that $|\widetilde\nabla h| > |\widetilde\nabla f|$, so \eqref{eq:nabla-f1-log-big} gives $\rho(x,a) > r/\varepsilon_n\to\infty$, as $n\to\infty$, contradiction.
\end{proof}

Our next aim is to give more precise localisation.

\begin{prop}
    \label{pro:cont-case}
    Let $(M,\rho)$ be a complete metric space and let $f:M\to\mathbb{R}\cup\{+\infty\}$ be a proper lower semicontinuous and bounded below function. Let $g:M\to\mathbb{R}$ be a continuous function such that
    \begin{equation}\label{eq:strict-compare-2}
        |\nabla f|(x) > |\nabla g|(x),\quad\forall x\in\dom f\setminus 0\crit f.
    \end{equation}
    Then for each $\varepsilon > 0$
    $$
        f - g \ge \inf_{\varepsilon\crit f}(f-g),
    $$
    or, more precisely, for each $x_0\in M$ there exists $x\in \varepsilon\crit f$ such that
    \begin{equation}
        \label{eq:cont-local-1}
        f(x)\le f(x_0) - \varepsilon\rho(x,x_0),
    \end{equation}
    and
    \begin{equation}
        \label{eq:cont-local-2}
        f(x_0) - g(x_0) \ge f(x)-g(x).
    \end{equation}
\end{prop}

\begin{proof}
    The claim is clear if $f(x_0) = \infty$, so let $x_0\in \dom f$. Let $\lambda := f(x_0) - g(x_0)$ and
    $$
        M_1 := L_\lambda(f-g),\quad f_1 := f\restriction_{M_1}.
    $$
    By Proposition~\ref{pro:EVP} applied to $f_1$ on $M_1$ there exists $x\in M_1$, that is, satisfying \eqref{eq:cont-local-2}; such that \eqref{eq:cont-local-1} is satisfied and $x\in\varepsilon\Crit f_1$.

    If $|\nabla f| (x) = 0$ we are done, because $x\in \varepsilon\crit f$.

    If $|\nabla f|(x) > 0$ we apply Lemma~\ref{lem:invar} to get $|\nabla f|(x) = |\nabla f_1|(x)$. Since $x\in\varepsilon\Crit f_1$ it follows that $|\nabla f_1|(x) \le \varepsilon$, so $|\nabla f|(x) \le \varepsilon$ and we are done.
\end{proof}
The  statement below is in the style of \cite{Aris-new}, but there an abstract approach is used, so it would be interesting to check how far our approach here goes.
The latter, however, is outside the scope of this article.
\begin{theorem}
 \label{thm:cont-case-new}
    Let $(M,\rho)$ be a complete metric space and let $f:M\to\mathbb{R}\cup\{+\infty\}$ be a proper lower semicontinuous and bounded below function. Let $g:M\to\mathbb{R}$ be a continuous function such that \eqref{eq:strict-compare-2} holds.

   For any $x_0\in \dom f$ at least one of the following two assertions holds:

   (a) There exists $\ol x\in 0\crit f$ such that
   \[
   f(x_0)-g(x_0)\ge f(\ol x)-g(\ol x);
 \]

   (b) For every sequence $\eps_n \searrow 0$ there  exists a sequence $x_n\in \eps_n\crit f$ which has no convergent subsequences, and satisfies
   \be\label{eq:ala}
   \sum_{n=1}^\infty \eps_n\rho(x_n,x_{n-1})<\infty,
   \ee
   and
   \be\label{eq:bala}
f(x_n)-g(x_n)\ge f(x_{n+1})-g(x_{n+1}),\quad \forall n\in\{0,\N\}.
   \ee
\end{theorem}
Note that \eqref{eq:ala} implies
$$
	\sum_{n=1}^\infty |\nabla f|(x_n)\rho(x_n,x_{n-1})<\infty,
$$
which is the form used in \cite{Aris-new}.
\begin{proof}
Assume first that (a) is false.

Fix a sequence $\eps_n \searrow 0$ and proceed by induction.

\textsc{Basic step:} Set $\la_0:= f(x_0)-g(x_0)$, $M_0:= L_{\la_0}(f-g)$, $f_0:= f\restriction_{M_0}$ and take $x_1\in \eps_1\Crit f_0$ such that
\[
f(x_1)\le f(x_0)-\eps_1\rho(x_1,x_0)
\]
and
\[
f(x_0)-g(x_0)\ge f(x_1)-g(x_1).
\]
The existence of such $x_1$ follows from  Proposition~\ref{pro:cont-case}.

 \textsc{Induction step:} Given $x_n\in \eps_n\Crit f_{n-1}$, set $\la_n:= f(x_n)-g(x_n)$, $M_n:=L_{\la_n}(f-g)$, $f_n:= f\restriction_{M_n}$ and take $x_n\in \eps_{n+1}\Crit f_n$ such that
\be\label{eq:mda}
f(x_{n+1})\le f(x_n)-\eps_{n+1}\rho(x_{n+1},x_n),
\ee
and
\[
f(x_n)-g(x_n)\ge f(x_{n+1})-g(x_{n+1}).
\]
The existence of  $x_{n+1}$ follows again from  Proposition~\ref{pro:cont-case}.

We claim that the sequence $\{ x_n\}_{n=1}^\infty$ has no convergent subsequence. To check this, assume the contrary.
Because $M_{n+1}\subset M_n$,  any subsequence of  the sequence $\{ x_n\}_{n=1}^\infty$ has the same structure, we may assume without loss of generality that
$\{ x_n\}_{n=1}^\infty$ itself converges. Let $x_n\to \ol x$.

Since
\[
f(x_0)-g(x_0)\ge f(x_n)-g(x_n),\quad \forall n\in \N,
\]
by continuity of $g$ at $\ol x$ and by lower semicontinuity of $f$ at $\ol x$, after passing to limit,  we obtain that
\[
f(x_0)-g(x_0)\ge f(\ol x)-g(\ol x).
\]
Since (a) is assumed false, $\ol x\not \in 0\crit f$, and therefore
\[
  |\nabla f|(\ol x) > |\nabla g|(\ol x).
  \]
 Now set $\ol\la:= f(\ol x)-g(\ol x)$, $\ol M: =L_{\ol\la}(f-g)$, $\ol f:= f\restriction_{\ol M}$.

 On the one hand, from Lemma~\ref{lem:invar} it follows that $  |\nabla \ol f|(\ol x) =  |\nabla f|(\ol x)$ and the above inequality yields $|\nabla \ol f|(\ol x)>0$.

 On the other hand, $\ol M\subset M_n$ and $\ol f=f_n\restriction_{\ol M}$ for all $n$. Fix $y\in \ol M$. Since $x_n\in \eps_n\Crit f_{n-1}$ we have that
 \[
 f(x_n)-f(y)\le \eps_n \rho(x_n,y),
 \]
  and passing to limit yields $f(\ol x )-f(y)\le 0$. Hence, $\ol x\in 0\Crit \ol f$, and in particular $|\nabla \ol f|(\ol x)=0$, which is a contradiction.

 Therefore,$\{ x_n\}_{n=1}^\infty$ has no convergent subsequences, \eqref{eq:bala} is fulfilled by construction, while for \eqref{eq:ala} note that by \eqref{eq:mda} we have
 \[
\eps_{n+1}\rho(x_{n+1},x_n) \le f(x_n)-f(x_{n+1}),\quad \forall n\in \{0,\N\},
 \]
 so
 \[
\sum_{n=0}^\infty \eps_{n+1}\rho(x_{n+1},x_n) \le f(x_0)-\inf f.
 \]
The proof is then complete.
 \end{proof}

Not surprisingly, Proposition~\ref{pro:cont-case}  has a compete analogue for the global slope.

\begin{prop}
    \label{pro:cont-case-tilde}
    Let $(M,\rho)$ be a complete metric space and let $f:M\to\mathbb{R}\cup\{+\infty\}$ be a proper lower semicontinuous and bounded below function. Let $g:M\to\mathbb{R}$ be a continuous function such that
    \begin{equation}\label{eq:strict-compare-2-tilde}
        |\widetilde\nabla f|(x) > |\widetilde\nabla g|(x),\quad\forall x\in\dom f\setminus 0\Crit f.
    \end{equation}
    Then for each $\varepsilon > 0$
    $$
        f - g \ge \inf_{\varepsilon\Crit f}(f-g),
    $$
    or, more precisely, for each $x_0\in M$ there exists $x\in \varepsilon\Crit f$ such that \eqref{eq:cont-local-1} and \eqref{eq:cont-local-2} hold.
\end{prop}
\begin{proof}
    One can repeat the  proof of Proposition~\ref{pro:cont-case} using Lemma~\ref{lem:invar-tilde}.
\end{proof}

\begin{corr}
    \label{cor:cont-case-tilda-le}
    Let $(M,\rho)$ be a complete metric space and let $f:M\to\mathbb{R}\cup\{+\infty\}$ be a proper lower semicontinuous and bounded below function. Let $g:M\to\mathbb{R}$ be a continuous function such that \eqref{eq:non-strict-compare-tilde} is fulfilled.
    Then for each $\varepsilon > 0$, $x_0\in M$  and $r\in(0,1)$ there exists $x\in \varepsilon\Crit f$ such that \eqref{eq:cont-local-1} holds and
    \begin{equation}
        \label{eq:r-cont-local-2}
        f(x_0) - rg(x_0) \ge f(x) - rg(x).
    \end{equation}
\end{corr}

\begin{proof}
    We apply Proposition~\ref{pro:cont-case-tilde} to $f$ and $rg$. It is clear that \eqref{eq:non-strict-compare-tilde} implies \eqref{eq:strict-compare-2-tilde} for $rg$, because $|\widetilde\nabla (rg)| = r|\widetilde\nabla g|$ and $0<r<1$.
\end{proof}

Now we can remove the continuity assumption, but once again, only for the global slope  $|\widetilde\nabla|$.
\begin{theorem}
    \label{thm:lsc-case}
    Let $(M,\rho)$ be a complete metric space and let $f:M\to\mathbb{R}\cup\{+\infty\}$ be  a proper lower semicontinuous and bounded below function. Let $g:M\to\mathbb{R}\cup\{+\infty\}$ be a lower semicontinuous function such that \eqref{eq:non-strict-compare-tilde} holds.

    Then for each $\varepsilon > 0$, $x_0\in \dom |\widetilde\nabla f|$  and $r\in(0,1)$ there exists $x\in \varepsilon\Crit f$ such that \eqref{eq:cont-local-1} and \eqref{eq:r-cont-local-2} hold. More precisely,
    \begin{equation}
        \label{eq:inf-f-rg}
        \inf_{\dom f}(f-rg) = \inf _{\varepsilon\Crit f} (f-rg).
    \end{equation}
\end{theorem}
\begin{proof}
    Let $x_0\in\dom |\widetilde\nabla f|$ and let $s := |\widetilde\nabla f|(x_0) < \infty$. We can also assume that $s>0$, because otherwise there is nothing to prove.

    Let $\lambda := f(x_0) - rg(x_0)$ and
    $$
        M_1 := s \Crit  f\cap L_\lambda(f-rg),\quad f_1 := f\restriction_{M_1},\quad g_1 := rg\restriction_{M_1}.
    $$
    Obviously, $x_0\in M_1$. From Lemma~\ref{lem:crit} we know that $s \Crit  f$ is closed and $f$ is Lipschitz on $s \Crit  f$. Also, \eqref{eq:non-strict-compare-tilde} implies that $s \Crit  f\subset s \Crit  g$ and thus $g$ is also Lipschitz on $s \Crit  f$. So, $(f-rg)\restriction_{s\Crit f}$ is Lipschitz and thus $L_\lambda(f-rg)\cap s\Crit f$ is closed.
    That is, $M_1$ is closed and $g_1$ is Lipschitz.

    It is clear that $|\widetilde\nabla g_1| \le r|\widetilde\nabla g|$ on $M_1$. On the other hand, Lemma~\ref{lem:invar-lsc} gives that $|\widetilde\nabla f_1| = |\widetilde\nabla f|$ on $M_1$. Therefore, \eqref{eq:non-strict-compare-tilde} translates to
    $$
        |\widetilde\nabla f_1|(x) > |\widetilde\nabla  g_1|(x),\quad\forall x\in M_1:\ |\widetilde\nabla f_1|(x) > 0,
    $$
    which is \eqref{eq:strict-compare-2-tilde} for $f_1$ and $g_1$. We can apply Proposition~\ref{pro:cont-case-tilde} to $f_1$ and $g_1$ and use once again that $|\widetilde\nabla f_1| = |\widetilde\nabla f|\restriction_{M_1}$.

    Now \eqref{eq:inf-f-rg} follows from Lemma~\ref{lem:inf-on-dom}.
\end{proof}

\section{Moreau-Rockafellar Theorem}
\label{sec:mrt}

We finally show that slopes provide quite different and modern approach to Moreau-Rockafellar Theorem. For a discussion on the known proofs see \cite{iz-pams}. In relation to this, note that the recent proof \cite{mat-zl} is also quite classical in flavour.

If $(X,\|\cdot\|)$ is a Banach space and $f:X\to \Ri$ is a convex function, its \emph{subdifferential} at $x_0\in \dom f$ is the   subset of the dual space $X^*$ given by
\[
\pa f(x_0):=\{ p\in X^*:p(x-x_0)\le f(x)-f(x_0),\ \forall x\in X\},
\]
and $\pa f(x_0)=\varnothing$ whenever $x_0\not \in \dom f$. As usual, $\dom \pa f:=\{ x:\pa f(x)\neq \varnothing\}$, see e.g.~\cite{Rockafellar-book1}.

For any convex function $f:X\to \Ri$ it is easy to see that the local slope and global slope coincide at any point.

\begin{theorem}[Moreau-Rockafellar]
  \label{thm:m-r}
  Let $(X,\|\cdot\|)$ be a Banach space and let $f,g:X\to\Ri$ be proper, convex and lower semicontinuous functions. If
\begin{equation}\label{eq:g-in-f}
	\partial g(x) \subset \partial f(x),\quad \forall x\in X,
\end{equation}
  then for some constant $c\in\R$
  $$
	f(x) = g(x) + c, \quad \forall x\in X.
  $$
\end{theorem}

\begin{proof}
From the maximal monotonicity of $\pa f$, see \cite{Rockafellar-maxmon} or \cite[Theorem 3.24]{Ph}, and \eqref{eq:g-in-f}
it follows that
\begin{equation}\label{eq:g-is-f}
	\partial g(x) \equiv \partial f(x),\quad \forall x\in X.
\end{equation}

Take $x_0\in \dom \pa g$ and take $p\in \pa g(x_0)$. So, $p$ satisfies
\[
g(x)\ge g(x_0)+p(x-x_0), \quad \forall x\in X.
\]

From \eqref{eq:g-in-f} it follows that $p\in \pa f(x_0)$, so
\[
f(x)\ge f(x_0)+p(x-x_0), \quad \forall x\in X.
\]

Set
\[
f_1(x):= f(x)-f(x_0)-p(x-x_0),\quad g_1(x):= g(x)-g(x_0)-p(x-x_0),\quad \forall x\in X.
\]

It is clear that $\min f_1=\min g_1=0$. From the Sum Theorem, see e.g. \cite[Theorem 3.16]{Ph}, and \eqref{eq:g-is-f} we have that
\begin{equation}\label{eq:g1-is-f1}
	\partial g_1(x) \equiv \partial f_1(x),\quad \forall x\in X.
\end{equation}

Because for the proper convex  lower semicontinuous function $f_1$ it holds that $\dom |\widetilde \nabla f_1|\equiv\dom |  \nabla f_1|\equiv\dom \pa f_1$ and
\[
|\widetilde \nabla f_1|(x)=|\nabla f_1|(x)=\min\{ \|p\| : p\in \pa f_1(x)\},\quad \forall x\in \dom \pa f_1,
\]
see \cite[Lemma 3.2]{TZ}, and the same is true for $g_1$, from \eqref{eq:g1-is-f1} it holds that
\[
|\widetilde \nabla f_1|(x)=|\widetilde \nabla g_1|(x)\quad \forall x\in \dom |\widetilde \nabla f_1|\equiv \dom |\widetilde \nabla g_1|.
\]

Since $\min f_1=\min g_1=0$, Theorem~\ref{thm:t-z} implies both $f_1(x)\ge g_1(x)$ and $g_1(x)\ge f_1(x)$ for all $x\in X$. That is, $f_1(x)=g_1(x)$ for all $x\in X$, which means that
\[
f(x)=g(x)+(f(x_0)-g(x_0)).
\]
\end{proof}

\bigskip

\noindent\textbf{Appendix.}
Even in trying to solve simple problems one is tempted to use transfinite induction.  Let us take for example the problem of finding the middle point of a segment  using only ruler and compass.

Let  the segment $[a,b]$ is given. Denote the distance between the points  $a$ and $b$ by $d(a,b)$. For finding the middle point $m$ consider the following transfinite procedure, \cite{Nikolay}:

\bigskip

\noindent
$l \leftarrow a$; $r \leftarrow b$

\noindent
till $l\neq r$:

\hspace{1cm}
    take arbitrary point $x$ in $[l,r]$ different from $l$ and $r$

\hspace{2cm}
        if $d(l,x)<d(x,r)$ then:

\hspace{3cm}
            $l\leftarrow x$;

\hspace{3cm}
            $r\leftarrow$ the point on $[l,r]$ at distance $d(l,x)$ to $r$;

\hspace{2cm}
        if $d(l,x)>d(x,r)$ then:

\hspace{3cm}
            $l\leftarrow$ the point on $[l,r]$ at distance $d(x,r)$ to $l$;

 \hspace{3cm}
            $r\leftarrow x$;

\hspace{2cm}
        if $d(l,x)=d(x,r)$ then stop ($x\equiv m$).

\bigskip

The segments $[l,r]$ are nested and their lengths monotonically decrease. Transferring the construction on the real line, the latter implies that we are bypassing at least one new rational number each time, meaning that at some countable ordinal we will get the middle point.

Although the above transfinite procedures works, in practice one uses the well known finite procedure of construction of the bisector.   Admittedly, the latter uses the plane, but it is anyway good to know.

\bigskip

\noindent\textbf{Acknowledgements.} Authors express their sincere gratitude for the fruitful discussions and exchange on the subject to Aris Daniilidis, Robert Deville, Alexander D. Ioffe, Lionel Thibault, and Dariusz Zagrodny.

\end{document}